\documentclass[11pt,a4paper]{article}
\usepackage{amssymb,amsmath}

\usepackage{graphicx}
\usepackage{amssymb}
\usepackage{epstopdf}
\newtheorem{theorem}{Theorem}
\newtheorem{corollary}[theorem]{Corollary}

\newtheorem{proposition}[theorem]{Proposition}
\newtheorem{remark}[theorem]{Remark}

\newenvironment{proof}{\begin{trivlist}\item[]{\it Proof.}}
{\hfill$\square$\end{trivlist}}

\def\cocoa{{\hbox{\rm C\kern-.13em o\kern-.07em C\kern-.13em o\kern-.15em A}}}
\def\mc{{\mathbb{C}}}
\def\mz{{\mathbb{Z}}}
\def\mn{{\mathbb{N}}}
\def\mq{{\mathbb{Q}}}

\begin{document}

\title{Defining relation for semi-invariants\\
of three by three matrix triples}
\author{M\'aty\'as Domokos ${}^a$
\thanks{Partially supported by OTKA NK72523, NK81203.}
\quad and \quad Vesselin Drensky ${}^b$
\\
\\
{\small ${}^a$ R\'enyi Institute of Mathematics,
Hungarian Academy of Sciences,}\\
{\small Re\'altanoda utca 13-15, 1053 Budapest, Hungary}\\
{\small E-mail: domokos@renyi.hu }
\\
\\
{\small ${}^b$ Institute of Mathematics and Informatics, Bulgarian Academy of Sciences,}\\
{\small  Acad. G. Bonchev Str., Block 8, 1113 Sofia, Bulgaria}\\
{\small E-mail: drensky@math.bas.bg}
}

\date{}                                           

\maketitle

\begin{abstract}
The single defining relation of the algebra of $SL_3\times SL_3$-invariants of triples of $3\times 3$
matrices is explicitly computed. Connections to some other prominent algebras of invariants are pointed out.
\end{abstract}


\section{Introduction}\label{sec:intro}

Denote by $M_n(K)$ the space of $n\times n$ matrices over an infinite field $K$.
The direct product of two copies of the general linear group $G_n:=GL_n(K)\times GL_n(K)$ acts linearly on $M_n(K)$:
The group element $(g,h)$  maps $A\in M_n(K)$ to $gAh^{-1}$. Take the direct sum
$M_n(K)^m:=\underbrace{M_n(K)\oplus\cdots\oplus M_n(K)}_m$ of $m$ copies of this representation of
$G_n$.
The action of $G_n$ induces an action on the algebra of polynomial functions $K[M_n(K)^m]$ in the usual way.
Let $R_{n,m}(K)$ be the subalgebra of the invariants of the subgroup $SL_n(K)\times SL_n(K)$ of $G_n$. It is called also
the algebra of semi-invariants of $G_n$ on $M_n(K)^m$. The structure and minimal systems of generators of $R_{n,m}(K)$ are known in few cases only.
Over a field of characteristic 0 or $p>2$ the algebra $R_{2,m}(K)$ is minimally generated by the determinants $\det(A_r)$,
$r=1,\ldots,m$, the mixed discriminants $M(A_{r_1},A_{r_2})$, $1\leq r_1<r_2\leq m$, and the discriminants
$D(A_{r_1},A_{r_2},A_{r_3},A_{r_4})$, $1\leq r_1<r_2<r_3<r_4\leq m$, \cite{domokos:2000}
or \cite[Theorem 11.47]{mukai}.
Here $M(A_1,A_2)$ is defined as the coefficient of $t_1t_2$ in
\[
\det(t_1A_1+t_2A_2)=t_1^2\det(A_1^2)+t_1t_2M(A_1,A_2)+t_2^2\det(A_2),
\]
\[
D(A_1,A_2,A_3,A_4)=\left\vert\begin{matrix}
a_{11}^{(1)}&a_{11}^{(2)}&a_{11}^{(3)}&a_{11}^{(4)}\\
&&&\\
a_{21}^{(1)}&a_{21}^{(2)}&a_{21}^{(3)}&a_{21}^{(4)}\\
&&&\\
a_{12}^{(1)}&a_{12}^{(2)}&a_{12}^{(3)}&a_{12}^{(4)}\\
&&&\\
a_{22}^{(1)}&a_{22}^{(2)}&a_{22}^{(3)}&a_{22}^{(4)}\\
\end{matrix}\right\vert,
\]
where $A_r=\left(a_{ij}^{(r)}\right)_{2\times 2}$, $r=1,2,3,4$.
For $m\leq 4$ the generators $\det(A_r)$ and $M(A_{r_1},A_{r_2})$
are algebraically independent and for $m=4$ the algebra $R_{2,4}(K)$
is a free module over the polynomial subalgebra generated by them
with basis $1,D(A_1,A_2,A_3,A_4)$.
It is pointed out in  \cite{domokos-drensky}, \cite{mukai} that $R_{2,m}(K)$ can be interpreted
as the ring of vector invariants of the special orthogonal group of degree $4$.
Therefore the relations among the generators can be deduced from classically known results,
and even a Gr\"obner basis of the ideal of relations can be obtained from \cite{domokos-drensky}.
For $m=2$ and any $n$ the algebra $R_{n,2}$ is generated by the algebraically independent
coefficients of $\det(t_1A_1+t_2A_2)$, \cite{happel}, see also \cite{koike}.

Apart from the cases $n=2$ for any $m$ and $m=2$ for any $n$, the only other case when a minimal system
of  generators of $R_{n,m}(K)$ is explicitly known is $n=m=3$, and this algebra is the main object to study in the present paper.
In the sequel we denote
$R(K):=K[M_3(K)^3]^{SL_3\times SL_3}$, the algebra
of $SL_3(K)\times SL_3(K)$-invariant polynomial functions on $M_3(K)^3$.

Define polynomial functions $f_{i,j,k}$ on $M_3(K)^3$ by the equality
\[
\det(t_1A_1+t_2A_2+t_3A_3)=\sum_{i+j+k=3}t_1^it_2^jt_3^kf_{i,j,k}(A_1,A_2,A_3)
\]
for all $t_1,t_2,t_3\in K$ and $A_1,A_2,A_3\in M_3(K)$. Obviously
the ten polynomials $f_{i,j,k}$ belong to $R(K)$. Furthermore,
define $h$ as the coefficient of $t_1^2t_2^2t_3^2$ in
\[
\det\left(\begin{array}{cc}t_2A_2 & t_1A_1 \\t_1A_1 & t_3A_3\end{array}\right)
\]
and define $q$ as the coefficient of $t_1^2t_2t_3^2t_4t_5^2t_6$ in
\[\det\left(\begin{array}{ccc}0 & t_1A_1 & t_2A_2 \\t_4A_1 & 0 & t_3A_3 \\t_5A_2 & t_6A_3 & 0\end{array}\right).\]
Clearly $h$ and $q$ belong to $R(K)$. It is proved in
\cite{domokos} that $h$ and the ten polynomials $f_{i,j,k}$ (where $i+j+k=3$) constitute
a homogeneous system of parameters in $R(K)$. Denote by $P(K)$ the subalgebra generated
by these eleven algebraically independent elements.  In the case when the characteristic of the base field $K$ is zero,
using a result of Teranishi \cite{teranishi} it was established in \cite{domokos} that $R(K)$ is a free $P(K)$-module generated by $1$ and $q$:
\begin{equation}\label{eq:freemodule}
R(K)=P(K)\oplus P(K) q.
\end{equation}
 A similar description of $R(K)$ is stated by Mukai without proof in  \cite[Proposition 11.49]{mukai}.
It follows from (\ref{eq:freemodule}) that $q$ satisfies a monic quadratic relation with
coefficients from $P(K)$.
In the present paper we find the explicit form of this relation.

A crucial role in our considerations is played by the following right action of
the general linear group $GL_3(K)$ on $M_3(K)^3$:
For $g=(g_{ij})_{3\times 3}\in GL_3(K)$ and $(A_1,A_2,A_3)\in M_3(K)^3$ we have
\[
(A_1,A_2,A_3)\cdot g:=\left(\sum_{i=1}^3g_{i1}A_i,\sum_{i=1}^2g_{i2}A_i,\sum_{i=1}^3g_{i3}A_i\right).
\]
This induces a left  action of $GL_3(K)$ on the coordinate ring of $M_3(K)^3$: For a polynomial function
$f$ on $M_3(K)^3$ and $g\in GL_3(K)$, the function $g\cdot f$ maps
$(A_1,A_2,A_3)\in M_3(K)^3$ to
$f((A_1,A_2,A_3)\cdot g)$.
Since this action of $GL_3(K)$ commutes with the action of $G$ introduced above,
$R(K)$ is a $GL_3(K)$-submodule of the coordinate ring of $M_3(K)^3$.

First in Section~\ref{sec:charzero} we treat the case when $K$ is the field $\mq$ of rational numbers.
By the theory of polynomial representations of $GL_3(\mq)$
one can read off from (\ref{eq:freemodule}) that $h$ and $q$ can be replaced by
$H$ and $Q$ that are {\it highest weight vectors} with respect to $GL_3(\mq)$. In fact $H$ and
$Q$ are invariants with respect to the subgroup $SL_3(\mq)$ of $GL_3(\mq)$ and they are
uniquely determined up to non-zero scalar multiples. The relation among the new generators
$Q,H,f_{i,j,k}$ takes place in the subalgebra of $SL_3(\mq)$-invariants in $R(\mq)$. This is a
``small'' subalgebra of $R(\mq)$, and a ``large'' part of it  can be identified with the algebra of
$SL_3(\mc)$-invariants of ternary cubic forms, whose explicit generators $S$ and $T$ are known from a
famous classical computation of Aronhold \cite{aronhold}.
It is an easy matter to find $H$ and $Q$ explicitly, and then most of the computational difficulty in
finding the relation among $Q,H,f_{i,j,k}$ is already contained in Aronhold's computation, so one gets
 easily the desired relation (cf. Theorem~\ref{thm:charzero}).

Rewriting the relation found in Section~\ref{sec:charzero} in terms of our original generators
$q$,$h$, $f_{i,j,k}$, we obtain a relation $A(q,h,f_1,\ldots,f_{10})=0$
with integer coefficients.
This yields a uniform description for $R(K)$ in terms of a minimal generating system and the corresponding defining relations,
valid over any infinite base field $K$ and also for $K=\mz$, the ring of integers, see Theorem~\ref{thm:charfree}.

The results in Theorems~\ref{thm:charzero} and ~\ref{thm:charfree} can be applied to recover in a transparent way
known results in three other topics of independent interest.
In Remark~\ref{remark:jacobian} we mention the connection to the explicit determination of the Jacobian of a cubic curve, and to the description of
$SL_3(\mc)\times SL_3(\mc)\times SL_3(\mc)$-invariants of
tensors in $\mc^3\otimes \mc^3\otimes \mc^3$.
Furthermore, in Section~\ref{sec:conjugation}
we deduce from Theorem~\ref{thm:charfree} the explicit combinatorial description of the ring of conjugation invariants of pairs of $3\times 3$ matrices.
In particular, we recover the complicated relation due to Nakamoto \cite{nakamoto}
as a simple consequence of our results on $R(K)$.
In summary, the complicated relation mentioned above comes from the
simple relation in Theorem~\ref{thm:charzero} by specialization and change of variables.

Let us note finally that $R$ is an instance of a semi-invariant algebra of a quiver, and
Theorems~\ref{thm:charzero} and ~\ref{thm:charfree} give information on the homogeneous coordinate ring of the moduli space of semistable
$(3,3)$-dimensional representations (cf. \cite{king}) of the generalized Kronecker quiver with three arrows.


\section{Characteristic zero}\label{sec:charzero}

Throughout this section we assume that $K=\mq$, the field of rational numbers. (Everything
would hold for any characteristic zero base field.) To simplify notation set
$R:=R(\mq)$, $P:=P(\mq)$.
The homogeneous components of $R$ are polynomial $GL_3(\mq)$-modules.
Recall that given a representation of $GL_3(\mq)$ on some vector space and $\alpha\in\mz^3$ we say that a non-zero vector $v$
is a weight vector of weight $\alpha$ if
${\mathrm{diag}}(z_1,z_2,z_3)\cdot v=z_1^{\alpha_1}z_2^{\alpha_2}z_3^{\alpha_3}v$ for all
diagonal elements ${\mathrm{diag}}(z_1,z_2,z_3)\in GL_3(\mq)$.
A polynomial $GL_3(\mq)$-module is completely reducible, and the isomorphism classes of irreducible polynomial
$GL_3(\mq)$-modules are labeled by partitions $\lambda$ with at most three non-zero parts, i.e.,
$\lambda=(\lambda_1,\lambda_2,\lambda_3)\in\mn_0^3$ is a triple of non-negative integers with
$\lambda_1\geq\lambda_2\geq\lambda_3$.
Write $V_{\lambda}$ for the irreducible polynomial $GL_3(\mq)$-module corresponding to
$\lambda$. Given a polynomial representation of $GL_3(\mq)$, a weight vector is called a {\it highest weight vector}
if it is fixed by all unipotent upper triangular elements in $GL_3(\mq)$.
Then its weight is necessarily a partition $\lambda$, and it generates a $GL_3(\mq)$-submodule isomorphic to $V_{\lambda}$.

The $GL_3(\mq)$-module structure of $R$ is encoded in its
$3$-variable Hilbert series
\[
H(R;t_1,t_2,t_3):=\sum_{\alpha\in \mn_0^3}\dim_{\mq}(R_{\alpha})t_1^{\alpha_1}t_2^{\alpha_2}t_3^{\alpha_3}\in\mz[[t_1,t_2,t_3]],
\]
where $R_{\alpha}$ denotes the $\alpha$ weight subspace of $R$. From 
(\ref{eq:freemodule}) we know that
\[
H(R;t_1,t_2,t_3)
=\frac{1+t_1^3t_2^3t_3^3}{(1-t_1^2t_2^2t_3^2)\prod_{i+j+k=3}(1-t_1^it_2^jt_3^k)}.
\]
This shows that up to degree $\leq 6$, the homogeneous components of $R$ coincide with those of $P$.
Hence $P$ is a $GL_3(\mq)$-submodule in $R$. Denote by $P_0$ the subalgebra of $P$ generated by the $f_{i,j,k}$.
For $g=(g_{ij})_{3\times 3}\in GL_3(\mq)$ we have
\begin{eqnarray*}
\sum_{i+j+k=3}t_1^it_2^jt_3^k f_{i,j,k}((A_1,A_2,A_3)\cdot g)=
\det\left(\sum_{j=1}^3\left(t_j\sum_{i=1}^3g_{ij}A_i\right)\right)
=\det\left(\sum_{i=1}^3\left(\sum_{j=1}^3g_{ij}t_j\right)A_i\right)\\
=\sum_{l+m+n=3}f_{l,m,n}(A_1,A_2,A_3)\left(\sum_{r=1}^3g_{1r}t_r\right)^l\left(\sum_{r=1}^3g_{2r}t_r\right)^m
\left(\sum_{r=1}^3g_{3r}t_r\right)^n,
\end{eqnarray*}
hence
\begin{equation}\label{eq:actiononfijk}
\sum_{i+j+k=3}(g\cdot f_{i,j,k})t_1^it_2^jt_3^k
=\sum_{l+m+n+3}f_{l,m,n}(A_1,A_2,A_3)\left(\sum_{r=1}^3g_{1r}t_r\right)^l\left(\sum_{r=1}^3g_{2r}t_r\right)^m
\left(\sum_{r=1}^3g_{3r}t_r\right)^n.
\end{equation}
So the $f_{i,j,k}$ span a $GL_3(\mq)$-submodule, hence $P_0$ is also a $GL_3(\mq)$-submodule,
and we see from the Hilbert series that the degree $6$ homogeneous component of $P_0$ has a $GL_3(\mq)$-module
direct complement in the degree $6$ homogeneous component of $P$ isomorphic to  $V_{(2,2,2)}$.
Taking the multidegree into account we conclude that there exist unique scalars $\beta_i$, $i=1,2,3,4$, such that
$H:=h+\beta_1f_{2,1,0}f_{0,1,2}+\beta_2f_{2,0,1}f_{0,2,1}+\beta_3 f_{1,2,0}f_{1,0,2}+
\beta_4 f_{1,1,1}^2$ is a highest weight vector (and hence spans the submodule $V_{(2,2,2)}$
mentioned above). To find the values $\beta_i$ note that
\[
\left(\begin{array}{ccc}1 & 1 & 0 \\0 & 1 & 0 \\0 & 0 & 1\end{array}\right)\quad
\mbox{ and }\quad \left(\begin{array}{ccc}1 & 0 & 0 \\0 & 1 & 1 \\0 & 0 & 1\end{array}\right)
\]
generate a Zariski dense subgroup in the subgroup of unipotent upper triangular matrices in $GL_3(\mq)$.
Therefore the condition that $H$ is a highest weight vector is equivalent to the condition that the above two elements of $GL_3(\mq)$ fix $H$.
This gives a system of linear equations for the $\beta_i$, that can be easily solved
(we used {\cocoa} \cite{cocoa}), and we get that
\begin{equation}\label{eq:H}
H=h-\frac 13f_{2,1,0}f_{0,1,2}-\frac 13f_{2,0,1}f_{0,2,1}+\frac 23 f_{1,2,0}f_{1,0,2}+
\frac 1{12}f_{1,1,1}^2.
\end{equation}
Denote by $P_+$ the sum of the positive degree homogeneous components of $P$.
Then by the above considerations we know that $P_+R$ is a $GL_3(\mq)$-submodule in
$R$, and the Hilbert series of $R$ shows that $P_+R$ has a $GL_3(\mq)$-module direct complement isomorphic
to $V_{(0)}\oplus V_{(3,3,3)}$ (where $V_{(0)}$ is the trivial
$GL_3(\mq)$-module). Consequently, there is a unique weight vector $v$ in $P$ of weight $(3,3,3)$
(i.e., a multihomogeneous element of multidegree $(3,3,3)$)
such that $Q:=q+v$ is a highest weight vector (and hence spans the submodule $V_{(3,3,3)}$ mentioned above).
Solving a small system of linear equations as in the case of $H$ we obtain
\begin{eqnarray}\label{eq:Q}
Q=&{}& q-\frac 12 h f_{1,1,1}+\frac 32f_{3,0,0}f_{0,3,0}f_{0,0,3}
\\ \notag &{}& -\frac 12f_{3,0,0}f_{0,2,1}f_{0,1,2}-\frac 12 f_{0,3,0}f_{2,0,1}f_{1,0,2}-\frac 12 f_{0,0,3}f_{2,1,0}f_{1,2,0}
\\ \notag &{}& -\frac 12f_{1,1,1}f_{1,2,0}f_{1,0,2}+\frac 12f_{2,1,0}f_{1,0,2}f_{0,2,1}+\frac 12f_{1,2,0}f_{2,0,1}f_{0,1,2}.
\end{eqnarray}

It follows from (\ref{eq:H}), (\ref{eq:Q}), (\ref{eq:freemodule}) that $Q$, $H$ and the $f_{i,j,k}$ constitute a minimal generating system of $R$,
and that $Q$ satisfies a monic quadratic polynomial with coefficients in $P$.
Note that $Q,H\in R^{SL_3(\mq)}$.
Next we introduce two other distinguished elements $\widetilde{S}$, $\widetilde{T}$ in
$R^{SL_3(\mq)}$.
Equation (\ref{eq:actiononfijk}) shows  that the algebraically independent invariants $f_{i,j,k}$ span
a $GL_3(\mq)$-submodule in $R$ isomorphic to the dual of the space of ternary cubic forms, hence
$P_0^{SL_3(\mq)}$ is isomorphic to the algebra of $SL_3(\mq)$-invariants of ternary cubic forms.
The latter was determined in \cite{aronhold}, and is generated by two algebraically independent elements $S$ and $T$.
Here $S$ and $T$ are homogeneous polynomials of degree four and six in the coefficients of the
general cubic ternary form
\[aX^3+bY^3+cZ^3+3a_2X^2Y+3a_3X^2Z+3b_1XY^2+3b_3Y^2Z+3c_1XZ^2+3c_2YZ^2+6mXYZ.\]
The expressions $S$ and $T$ can be found  in \cite{salmon},  in \cite[page 160]{dolgachev}, or in \cite{anetal}.
Now substitute in $S$ and $T$ the coefficients of the general ternary form by the $f_{i,j,k}$ to get
elements  $\widetilde{S}$ and $\widetilde{T}$ in $P_0$. The exact substitution is
given by the following table:
\[
\begin{tabular}{c|c|c|c|c|c|c|c|c|c|c}
$a$ & $a_2$ & $a_3$ & $b$ &$ b_1$ & $b_3$  & $c$ & $c_1$ & $c_2$ & $m$ \\
\hline
$f_{3,0,0}$ & $\frac 13f_{2,1,0}$ & $\frac 13f_{2,0,1}$ & $f_{0,3,0}$ & $\frac 13f_{1,2,0}$ &
$\frac 13f_{0,2,1}$ & $f_{0,0,3}$ & $\frac 13f_{1,0,2}$ & $\frac 13f_{0,1,2}$ &
$\frac 16f_{1,1,1}$
\end{tabular} \]
By (\ref{eq:actiononfijk}) this substitution induces a $GL_3(\mq)$-module isomorphism from the dual
of the space of ternary cubic forms to the subspace of $R$ spanned by the $f_{i,j,k}$. Hence
$\widetilde{S}$, $\widetilde{T}$ are $SL_3(\mc)$-invariants in $P_0$, and
by \cite{aronhold} we have  $P_0^{SL_3(\mc)}=\mc[\widetilde{S},\widetilde{T}]$.
Moreover, since $H$ is $SL_3(\mq)$-invariant, we have
$P^{SL_3(\mq)}=(P_0[H])^{SL_3(\mq)}=P_0^{SL_3(\mq)}[H]=\mq[\widetilde{S},\widetilde{T},H]$, a three-variable polynomial algebra.
Since $Q$ is also $SL_3(\mq)$-invariant, we conclude from
$R=P\oplus P\cdot Q$ that
\[
R^{SL_3(\mq)}=P^{SL_3(\mq)}\oplus Q\cdot P^{SL_3(\mq)}=
\mq[\widetilde{S},\widetilde{T},H]\oplus Q\cdot \mq[\widetilde{S},\widetilde{T},H].
\]
Taking the degrees into account it follows that
$Q^2=\alpha H^3+\beta H\widetilde{S}+ \gamma\widetilde{T}$ for some unique scalars
$\alpha,\beta,\gamma\in\mq$.
The scalars  can be easily found by substituting special matrix triples into the above equality:
on skew-symmetric triples all the $f_{i,j,k}$ vanish, hence $\widetilde{T}$ and $\widetilde{S}$ vanish.
On the other hand, the value of $H$ on the triple
\[
\left(\left(\begin{array}{ccc}0 & -x_1 & -y_1 \\x_1 & 0 & -z_1 \\ y_1 &  z_1 & 0\end{array}\right),
\left(\begin{array}{ccc}0 & -x_2 & -y_2 \\ x_2 & 0 & -z_2 \\ y_2 &  z_2 & 0\end{array}\right),
\left(\begin{array}{ccc}0 & -x_3 & -y_3 \\ x_3 & 0 & -z_3 \\ y_3 & z_3 & 0\end{array}\right)\right)
\]
is
$\det^2\left(\begin{array}{ccc}x_1 & x_2 & x_3 \\y_1 & y_2 & y_3 \\z_1 & z_2 & z_3\end{array}\right)$, whereas the value of $Q$ on this triple is
$\det^3\left(\begin{array}{ccc}x_1 & x_2 & x_3 \\y_1 & y_2 & y_3 \\z_1 & z_2 & z_3\end{array}\right)$.
This shows that $\alpha=1$.
Note that
\[
\det\left(\begin{array}{ccc}t_1 & t_3 & 0 \\0 & at_1+t_2 & -bt_1+t_3 \\ bt_1+t_3 & 0 & -at_1+t_2\end{array}\right)
=t_3^3+t_2^2t_1-b^2t_1^2t_3-a^2t_1^3
\]
(the Weierstrass canonical form of a plane cubic in homogeneous coordinates $(t_1:t_2:t_3)$ on
${\mathbb{P}}^2$).
The values of the invariants $\widetilde{S}$, $\widetilde{T}$, $H$, $Q$ on
the corresponding matrix triple
\[
\left(\left(\begin{array}{ccc}1 & 0 & 0 \\ 0 & a & -b \\ b &  0 & -a\end{array}\right),
\left(\begin{array}{ccc}0 & 0 & 0 \\ 0 & 1  & 0 \\ 0 &  0 & 1\end{array}\right),
\left(\begin{array}{ccc}0 & 1 & 0 \\ 0 & 0 & 1 \\ 1 & 0 & 0\end{array}\right)\right)
\]
are $-b^2/27$, $-4a^2/27$, $-b$, $-a$. It follows that
$\beta=27$ and $\gamma=-27/4$.
Hence we proved the following:

\begin{theorem}\label{thm:charzero}
We have the equality
\begin{equation}\label{eq:quadrel}
Q^2=H^3+27H\widetilde{S}-\frac{27}{4}\widetilde{T}
\end{equation}
where $H$ and $Q$ are given explicitly in (\ref{eq:H}), (\ref{eq:Q}), and they are characterized
(up to non-zero scalar multiples) in $R$ as the unique degree $6$ and degree $9$
$SL_3(\mq)$-invariants in $R$.
\end{theorem}

\begin{remark}\label{remark:jacobian} {\rm
(i) The relation (\ref{eq:quadrel}) essentially coincides with the relation
\[J^2=4\Theta^3+108S\Theta H^4-27TH^6\]
among the basic covariants of plane cubics (cf. \cite{salmon}).
It was observed by Weil \cite{weil} that this gives the equation of the Jacobian of a plane cubic,
see \cite{anetal} for a proof.
As it is pointed out in the book of Mukai \cite[page 430]{mukai},  an alternative approach to this result can be based
on the study of $R$ and its defining relation  (\ref{eq:quadrel}).
 We mention that the results on the equation of the Jacobian of a plane cubic are extended to arbitrary characteristic
 (including $2$ and $3$) in \cite{artinetal}. Our results in Section~\ref{sec:integers} have relevance for this.

(ii) The algebra $R^{SL_3(\mq)}$ can be identified with the algebra of $SL_3(\mq)\times SL_3(\mq)\times SL_3(\mq)$-invariants
on $\mq^3\otimes\mq^3\otimes \mq^3$.
The arguments above show that this is a three-variable polynomial algebra generated by
$\widetilde{S}$, $H$, and $Q$. This result is well known, see \cite{chanler},  \cite{vinberg:1},
\cite{vinberg:2}, \cite{thibon}.  Our results provide an alternative proof, and an alternative interpretation of the basic invariants.}
\end{remark}

\section{Relation over the integers}\label{sec:integers}

To shorten the expressions, set
\begin{eqnarray*}
f_1:=f_{3,0,0},\  f_2:=f_{2,1,0}, \ f_3:=f_{2,0,1}, \ f_4:=f_{1,2,0}, \ f_5:=f_{1,1,1},\\
f_6:=f_{1,0,2}, \ f_7:=f_{0,3,0}, \ f_8:=f_{0,2,1}, \ f_9:=f_{0,1,2}, \ f_{10}:=f_{0,0,3}.
\end{eqnarray*}

It turns out that $Q^2-H^3-27H\widetilde{S}+\frac{27}{4}\widetilde{T}=A(q,h,f_1,\ldots,f_{10})$,
where $A$ is a $12$-variable polynomial with integer coefficients,
given explicitly as follows:
\begin{eqnarray*}
&{}&A(q,h,f_1,\ldots,f_{10})=
q^2 - qhf_5
\\&{}& + 3qf_1f_7f_{10} - qf_1f_8f_9 - qf_2f_4f_{10} + qf_2f_6f_8 + qf_3f_4f_9 - qf_3f_6f_7 - qf_4f_5f_6
\\&{}& - h^3 + h^2f_2f_9 + h^2f_3f_8 - 2h^2f_4f_6
\\&{}& + 3hf_1f_4f_8f_{10} - hf_1f_4f_9^2 - 6hf_1f_5f_7f_{10} + hf_1f_5f_8f_9 + 3hf_1f_6f_7f_9 - hf_1f_6f_8^2
\\&{}& - hf_2^2f_8f_{10} + 3hf_2f_3f_7f_{10} - hf_2f_3f_8f_9 + hf_2f_4f_5f_{10} + hf_2f_4f_6f_9 - hf_2f_6^2f_7
\\&{}& - hf_3^2f_7f_9 - hf_3f_4^2f_{10} + hf_3f_4f_6f_8 + hf_3f_5f_6f_7 - hf_4^2f_6^2 + 9f_1^2f_7^2f_{10}^2
\\&{}& - 6f_1^2f_7f_8f_9f_{10} + f_1^2f_7f_9^3 + f_1^2f_8^3f_{10} - 6f_1f_2f_4f_7f_{10}^2 + f_1f_2f_4f_8f_9f_{10}
\\&{}& + 3f_1f_2f_5f_7f_9f_{10} - f_1f_2f_5f_8^2f_{10} + 3f_1f_2f_6f_7f_8f_{10} - 2f_1f_2f_6f_7f_9^2
\\&{}& + 3f_1f_3f_4f_7f_9f_{10} - 2f_1f_3f_4f_8^2f_{10} + 3f_1f_3f_5f_7f_8f_{10} - f_1f_3f_5f_7f_9^2 - 6f_1f_3f_6f_7^2f_{10}
\\&{}& + f_1f_3f_6f_7f_8f_9 + f_1f_4^3f_{10}^2 - f_1f_4^2f_5f_9f_{10} + f_1f_4^2f_6f_8f_{10} + f_1f_4f_5^2f_8f_{10}
\\&{}& - 3f_1f_4f_5f_6f_7f_{10} + f_1f_4f_6^2f_7f_9 - f_1f_5^3f_7f_{10} + f_1f_5^2f_6f_7f_9 - f_1f_5f_6^2f_7f_8 + f_1f_6^3f_7^2
\\&{}& + f_2^3f_7f_{10}^2 - 2f_2^2f_3f_7f_9f_{10} + f_2^2f_3f_8^2f_{10} - f_2^2f_4f_6f_8f_{10} - f_2^2f_5f_6f_7f_{10} + f_2^2f_6^2f_7f_9
\\&{}& - 2f_2f_3^2f_7f_8f_{10} + f_2f_3^2f_7f_9^2 - f_2f_3f_4f_5f_8f_{10} + 4f_2f_3f_4f_6f_7f_{10} + f_2f_3f_5^2f_7f_{10}
\\&{}&  - f_2f_3f_5f_6f_7f_9 + f_2f_4^2f_5f_6f_{10} - f_2f_4f_6^3f_7 + f_3^3f_7^2f_{10} + f_3^2f_4^2f_8f_{10} - f_3^2f_4f_5f_7f_{10}
\\&{}& - f_3^2f_4f_6f_7f_9 - f_3f_4^3f_6f_{10} + f_3f_4f_5f_6^2f_7.
\end{eqnarray*}

For $i,j,k\in\{1,2,3\}$ denote by $x_{ij}^{(r)}$ the coordinate function on $M_3(\mq)^3$ mapping the
matrix triple $(A_1,A_2,A_3)$ to the $(i,j)$-entry of $A_r$. Then $R(\mq)$ contains the subring
 \[R(\mz):=R(\mq)\cap\mz[x_{ij}^{(r)}\mid i,j,r=1,2,3].\]

\begin{theorem}\label{thm:charfree}
Let $K$ be an infinite field or the ring of integers.
Then $R(K)$ is minimally generated as a $K$-algebra by the twelve elements $q,h,f_j$, $j=1,\ldots,10$,
satisfying the single algebraic relation
$A(q,h,f_1,\ldots,f_{10})=0$ (where $A$ is given explicitly above). Moreover, $R(K)$ is a free module with basis $1,q$
over its $K$-subalgebra generated by the eleven algebraically independent elements $h,f_1,\ldots,f_{10}$.
\end{theorem}

\begin{proof}
We know already from \cite{domokos} and  Section~\ref{sec:charzero} that the statement holds when $K$ is a field of characteristic zero.
We also know already that for any $K$, the given twelve elements satisfy the relation
$A(q,h,f_1,\ldots,f_{10})=0$.

Suppose next that $K$ is an infinite field of positive characteristic. We claim that $1$ and $q$ generate
a free $P(K)$-submodule in $R(K)$. Indeed, otherwise $q$ belongs to the field of fractions of $P(K)$.
By the above relation $q$ is integral over $P(K)$. Since $P(K)$ is a unique factorization domain,
it follows that $q$ belongs to $P(K)$. Taking the grading of $R$ into account, we conclude that
$q=hc+d$, where $c$ is a linear combination of the $f_i$, and $d$ is a cubic polynomial in the $f_i$.
Now substitute into this equality a triple $(A_1,A_2,A_3)$, where the $A_i$ constitute a basis of the
space of $3\times 3$ skew-symmetric matrices. All the $f_i$ vanish on this triple, hence $(hc+d)(A_1,A_2,A_3)=0$,
whereas $q$ does not vanish on this triple as we pointed out in Section~\ref{sec:charzero}.
So $P(K)\oplus P(K)q\subseteq R(K)$. It follows from the theory of modules with good filtration
(cf. \cite[page 399]{donkin}) that the Hilbert series of $R(K)$ coincides with the
Hilbert series of $R(\mq)$. We know already that the latter coincides with the Hilbert series of
$P(K)\oplus P(K)q$, hence we have the equality $R(K)=P(K)\oplus P(K)q$.
This shows both the statement on the generators and the relation.

Finally we turn to $R(\mz)$. Denote by $P(\mz)$ the $\mz$-subalgebra of $R(\mz)$ generated by the eleven elements $h,f_1,\ldots,f_{10}$.  From the case $K=\mq$ we know that $P(\mz)$ is a polynomial ring,   and $R(\mz)$ contains the free
$P(\mz)$-submodule $M:=P(\mz)\oplus P(\mz)q$. Take any $f\in R(\mz)$.
It follows from the case $K=\mq$ that some positive integer multiple $mf$ of $f$ belongs to $M$,
so $mf=c+dq$, where $c,d\in P(\mz)$. We may assume that $m$ is minimal. If $m\neq 1$, then let $p$ be a prime divisor of $m$,
and let $L$ be an infinite field of characteristic $p$.
Reduction mod $p$ of coefficients gives a ring homomorphism
$\pi:Z:=\mz[x_{ij}^{(r)}\mid i,j,r=1,2,3]\to L[M_3(L)^3]$, and this restricts to a ring homomorphism
$\pi:R(\mz)\to R(L)$ and $\pi:P(\mz)\to P(L)$. Since $\pi(mf)=0$, we get that $\pi(c)+\pi(d)q=0$
holds in $R(L)$. From the case $K=L$ of our Theorem we know that $1,q$ are independent over $P(L)$,
hence $\pi(c)=0$ and $\pi(d)=0$, i.e. $c,d\in P(\mz)\cap pZ$. Clearly $P(\mz)\cap pZ=pP(\mz)$, since the eleven generators
of $P(\mz)$ are mapped under $\pi$ to algebraically independent
elements of $R(L)$. Consequently, $(m/p)f\in M$, contradicting the minimality of $m$.
Thus we have proved the equality $R(\mz)=P(\mz)\oplus P(\mz)q$.
This implies both the statement on the generators of $R(\mz)$ and the statement on the relation.
\end{proof}

\begin{remark}\label{remark:accidental} {\rm
The fact that a minimal $\mz$-algebra generating system of $R(\mz)$ stays a minimal $K$-algebra
generating system of $R(K)$ when exchanging the base ring to any infinite field $K$ is accidental,
and the analogous property does not hold in general in similar situations.
For example, denote by $R_{n,m}(K)$ the ring of $SL_n\times SL_n$-invariants of $m$-tuples of
$n\times n$ matrices.
It is proved in \cite{domokos-zubkov} that the method we used to construct generators
in the special case $n=m=3$ (i.e. polarization of the determinant of block matrices) yields in general
an (infinite) generating system of $R_{n,m}(K)$ for any infinite base field $K$, hence also for $K=\mz$.
(The latter claim follows in the same way as it is explained  by Donkin \cite{donkin} in a related situation.)
However, Proposition~\ref{prop:specI} below and the results of \cite{domokos-kuzmin-zubkov} imply that if $m$
is sufficiently large, then a minimal $\mz$-algebra generating system of $R_{n,m}(\mz)$ becomes redundant over
fields $K$ whose characteristic is zero or greater than $n$.}
\end{remark}


\section{Conjugation invariants of pairs of $3\times 3$ matrices}\label{sec:conjugation}

The general linear group $GL_3(K)$ acts on $M_3(K)^2=M_3(K)\oplus M_3(K)$ by simultaneous conjugation:
For $g\in GL_3(K)$ and $A,B\in M_3(K)$ we set
$g \cdot (A,B)=(gAg^{-1},gBg^{-1})$.
For any infinite field $K$ denote by $U(K):=K[M_3(K)^2]^{GL_3(K)}$ the corresponding algebra of invariants.
Similarly to Section~\ref{sec:integers}, consider
\[
U(\mz):=U(\mq)\cap\mz[x_{ij}^{(r)}\mid i,j=1,2,3;\ r=1,2]
\]
where $x_{ij}^{(r)}$ is the coordinate function assigning to the pair $(A_1,A_2)\in M_3(\mq)^2$ the $(i,j)$-entry of $A_r$.
A minimal system of generators of $U(K)$ was given by Teranishi \cite{teranishi} when
${\mathrm{char}}(K)=0$; Nakamoto \cite{nakamoto} extended the result for any infinite base field
$K$ or $K=\mathbb{Z}$, and determined
the single defining relation among the generators.
An exact description of $U(K)$ can also be obtained  from Theorem~\ref{thm:charfree},
using
the following statement (proved in \cite[Proposition 4.1]{domokos}):

\begin{proposition}\label{prop:specI}
The specialization $x_{ij}^{(3)}\mapsto \delta^i_j$ (where $\delta^i_j=1$ if $i=j$ and
$\delta^i_j=0$ otherwise) induces a surjection
$\varphi:R(K)\to U(K)$.
\end{proposition}

\begin{corollary}\label{cor:nakamoto}
Let $K$ be an infinite field or the ring of integers.
Then $U(K)$ is minimally generated as a $K$-algebra by the eleven elements $\varphi(q),
\varphi(h),\varphi(f_j)$, $j=1,\ldots,9$, satisfying the single algebraic relation
$A(\varphi(q),\varphi(h),\varphi(f_1),\ldots,\varphi(f_9),1)=0$ (where $A$ is given explicitly in Section~\ref{sec:integers}).
Moreover, $R(K)$ is a free module with basis $1,\varphi(q)$ over its
$K$-subalgebra generated by the ten algebraically independent elements $\varphi(h),\varphi(f_1),\ldots,\varphi(f_9)$.
\end{corollary}

\begin{proof} First we express the $\varphi$-images of the generators of $R(K)$ in terms of the usual generators of $U(K)$.
Define the functions $t,s,d$ on $M_3(K)$ by the equality
\[
\det(zI+A)=z^3+t(A)z^2+s(A)z+d(A),
\]
where $I$ is the $3\times 3$ identity matrix and $z\in K$ arbitrary.
One has the equality
\[s(AB)=t(A^2B^2)+t(AB)t(A)t(B)-t(A^2B)t(B)-t(AB^2)t(A)-s(A)s(B)\]
for $A,B\in M_3(K)$ (see \cite[Lemma 2]{domokos:2002} for a generalization).
 Furthermore, we have
\begin{eqnarray*}&{}&\varphi(f_{3,0,0})(A,B)=d(A),\ \varphi(f_{0,3,0})(A,B)=d(B),\
\varphi(f_{0,0,3})(A,B)=1,
\\ &{}& \varphi(f_{2,0,1})(A,B)=s(A),\qquad \varphi(f_{0,2,1})(A,B)=s(B), \\
&{}& \varphi(f_{1,0,2}(A,B)=t(A),\qquad \varphi(f_{0,1,2})(A,B)=t(B), \\
&{}&  \varphi(f_{1,1,1})(A,B)=t(A)t(B)-t(AB), \\
&{}&  \varphi(f_{2,1,0})(A,B)=t(A^2B)-t(AB)t(A)+s(A)t(B),\\
&{}&  \varphi(f_{1,2,0})(A,B)=t(AB^2)-t(AB)t(B)+t(A)s(B).
\end{eqnarray*}
Applying Amitsur's formula \cite{amitsur} one gets
\begin{eqnarray*}
\varphi(h)(A,B)&=&-t(A^2B^2)+t(A^2B)t(B)-t^2(A)s(B)+2s(A)s(B), \\
\varphi(q)(A,B)&=&t(B^2A^2BA)-s(A)s(B)t(AB)-t(A^2B)t(AB)t(B) \\
&{}& -t(AB^2)t(AB)t(A)+t^2(AB)t(A)t(B).
\end{eqnarray*}
Therefore by Theorem~\ref{thm:charfree} and Proposition~\ref{prop:specI},
the eleven elements
\begin{equation}\label{eq:generators}
t(A),s(A),d(A),t(B),s(B),d(B),t(AB),t(A^2B),t(AB^2),t(A^2B^2),t(B^2A^2BA)
\end{equation}
generate $U(K)$. Moreover, $t(B^2A^2BA)$ satisfies a monic quadratic relation over the subalgebra $W(K)$ of $U(K)$
generated by the first ten elements. Since by general principles on group actions, the transcendence degree of $U(K)$ is ten,
the first ten generators are algebraically independent. Moreover, $t(B^2A^2BA)$ does not vanish on the
pair $(E_{21}-E_{32},E_{12}+E_{23})$ (where $E_{ij}$ is the matrix unit whose only non-zero entry is a $1$ in the $(i,j)$-position),
whereas all the first nine generators vanish on this pair. Since the tenth generator has degree $4$ and the eleventh generator
has degree $6$, it follows that
 $t(B^2A^2BA)$ is not contained in $W(K)$. Hence by the integral closeness of $W(K)$ and by the quadratic relation we conclude
 that $U(K)=W(K)\oplus t(B^2A^2BA) W(K)$.
 The statements in our Corollary obviously follow.
\end{proof}

\begin{corollary}\label{cor:conjinvcharzero}
When $K$ is an infinite field with $\mathrm{char}(K)\neq 2$ or $3$, then the algebra $U(K)$ is minimally generated by
$\varphi(Q),\varphi(H),\varphi(f_1),\ldots,\varphi(f_9)$, and these generators satisfy the single algebraic relation
\[\varphi(Q)^2=\varphi(H)^3+27\varphi(H)\varphi(\widetilde{S})-\frac{27}{4}\varphi(\widetilde{T}).\]
\end{corollary}

\begin{remark} {\rm
(i) Expressing the left-hand side of $A(\varphi(q),\varphi(h),\varphi(f_1),\ldots,\varphi(f_9),1)=0$ in terms of the generators (\ref{eq:generators})
we obtain a transparent derivation of the relation found originally by hard computational labour by Nakamoto \cite{nakamoto}.
We include this relation in the Appendix.

(ii) The form of the relation in Corollary~\ref{cor:conjinvcharzero} is rather simple (or better to say that the complication is built into the
nineteenth century  expressions for $S$ and $T$ due to \cite{aronhold}): Indeed, the quartic or  sextic generators appear only
in three terms, and the remaining $9$ generators appear only in two
prominent classically known (though complicated) expressions.

(iii) We note that working over a characteristic zero base field, another minimal generating system of $U(K)$
is found in \cite{aslaksen-drensky-sadikova}, such that the relation between them takes a  simpler form than the relation
 in \cite{nakamoto}. This relation is of different nature than
the one in Corollary~\ref{cor:conjinvcharzero}. }
\end{remark}

\pagebreak \noindent {\bf Appendix.}  Setting
\begin{eqnarray*} t_1:=t(A),\quad s_1:=s(A),\quad d_1:=d(A),\quad t_2:=t(B), \quad s_2:=s(B), \quad d_2:=d(B)\\
r:=t(B^2A^2BA), \quad k:=t(A^2B^2),\quad w_1:=t(A^2B),\quad w_2:=t(AB^2),\quad z:=t(AB)
\end{eqnarray*}
we have
\begin{eqnarray*}
0&=&r^2 - rkz + rkt_1t_2 - rw_1w_2 - rw_1t_1t_2^2 - rw_2t_1^2t_2
\\ &+& rzt_1^2t_2^2 + 3rd_1d_2 - rd_1s_2t_2 - rd_2s_1t_1 - rs_1s_2t_1t_2
\\ &+& k^3 - 2k^2w_1t_2 - 2k^2w_2t_1 + k^2zt_1t_2 - 5k^2s_1s_2 + k^2s_1t_2^2 + k^2s_2t_1^2
\\ &+& kw_1^2s_2 + kw_1^2t_2^2 + kw_1w_2z + 2kw_1w_2t_1t_2 - kw_1zs_2t_1 - kw_1zt_1t_2^2 - 3kw_1d_2s_1
\\ &+& kw_1d_2t_1^2 + 9kw_1s_1s_2t_2 - 2kw_1s_1t_2^3 - 2kw_1s_2t_1^2t_2 + kw_2^2s_1 + kw_2^2t_1^2 - kw_2zs_1t_2
\\ &-& kw_2zt_1^2t_2 - 3kw_2d_1s_2 + kw_2d_1t_2^2 + 9kw_2s_1s_2t_1 - 2kw_2s_1t_1t_2^2 - 2kw_2s_2t_1^3 + kz^2s_1s_2
\\ &-& 6kzd_1d_2 + 4kzd_1s_2t_2 - kzd_1t_2^3 + 4kzd_2s_1t_1 - kzd_2t_1^3 - 8kzs_1s_2t_1t_2 + 2kzs_1t_1t_2^3
\\ &+& 2kzs_2t_1^3t_2 + 3kd_1d_2t_1t_2 - 2kd_1s_2^2t_1 - 2kd_2s_1^2t_2 + 8ks_1^2s_2^2 - 2ks_1^2s_2t_2^2 - 2ks_1s_2^2t_1^2
\\ &+& w_1^3d_2 - w_1^3s_2t_2 - w_1^2w_2s_2t_1 - 2w_1^2zd_2t_1 + 2w_1^2zs_2t_1t_2 + 4w_1^2d_2s_1t_2 - w_1^2d_2t_1^2t_2
\\&-& w_1^2s_1s_2^2 - 4w_1^2s_1s_2t_2^2 + w_1^2s_1t_2^4 + w_1^2s_2t_1^2t_2^2 - w_1w_2^2s_1t_2 + w_1w_2zs_1t_2^2 + w_1w_2zs_2t_1^2
\\ &-& 6w_1w_2d_1d_2 + 4w_1w_2d_1s_2t_2 - w_1w_2d_1t_2^3 + 4w_1w_2d_2s_1t_1 - w_1w_2d_2t_1^3 - 8w_1w_2s_1s_2t_1t_2
\\ &+& 2w_1w_2s_1t_1t_2^3 + 2w_1w_2s_2t_1^3t_2 + w_1z^2d_2s_1 + w_1z^2d_2t_1^2 - w_1z^2s_1s_2t_2 - w_1z^2s_2t_1^2t_2
\\ &+& 6w_1zd_1d_2t_2 + w_1zd_1s_2^2 - 4w_1zd_1s_2t_2^2 + w_1zd_1t_2^4 - 8w_1zd_2s_1t_1t_2 + 2w_1zd_2t_1^3t_2
\\ &+&  w_1zs_1s_2^2t_1 + 8w_1zs_1s_2t_1t_2^2 - 2w_1zs_1t_1t_2^4 - 2w_1zs_2t_1^3t_2^2 - 3w_1d_1d_2s_2t_1 - 2w_1d_1d_2t_1t_2^2
\\ &+& 2w_1d_1s_2^2t_1t_2 + 4w_1d_2s_1^2s_2 + 2w_1d_2s_1^2t_2^2 - w_1d_2s_1s_2t_1^2 - 8w_1s_1^2s_2^2t_2 + 2w_1s_1^2s_2t_2^3
\\ &+& 2w_1s_1s_2^2t_1^2t_2 + w_2^3d_1 - w_2^3s_1t_1 - 2w_2^2zd_1t_2 + 2w_2^2zs_1t_1t_2 + 4w_2^2d_1s_2t_1 - w_2^2d_1t_1t_2^2
\\ &-& w_2^2s_1^2s_2 - 4w_2^2s_1s_2t_1^2 + w_2^2s_1t_1^2t_2^2 + w_2^2s_2t_1^4 + w_2z^2d_1s_2 + w_2z^2d_1t_2^2 - w_2z^2s_1s_2t_1
\\ &-& w_2z^2s_1t_1t_2^2 + 6w_2zd_1d_2t_1 - 8w_2zd_1s_2t_1t_2 + 2w_2zd_1t_1t_2^3 + w_2zd_2s_1^2 - 4w_2zd_2s_1t_1^2
\\ &+& w_2zd_2t_1^4 + w_2zs_1^2s_2t_2 + 8w_2zs_1s_2t_1^2t_2 - 2w_2zs_1t_1^2t_2^3 - 2w_2zs_2t_1^4t_2 - 3w_2d_1d_2s_1t_2
\\ &-& 2w_2d_1d_2t_1^2t_2 + 4w_2d_1s_1s_2^2 - w_2d_1s_1s_2t_2^2 + 2w_2d_1s_2^2t_1^2 + 2w_2d_2s_1^2t_1t_2 - 8w_2s_1^2s_2^2t_1
\\  &+& 2w_2s_1^2s_2t_1t_2^2 + 2w_2s_1s_2^2t_1^3 + z^3d_1d_2 - z^3d_1s_2t_2 - z^3d_2s_1t_1 + z^3s_1s_2t_1t_2 - 5z^2d_1d_2t_1t_2
\\ &+& 4z^2d_1s_2t_1t_2^2 - z^2d_1t_1t_2^4 + 4z^2d_2s_1t_1^2t_2 - z^2d_2t_1^4t_2 - z^2s_1^2s_2^2 - 4z^2s_1s_2t_1^2t_2^2 + z^2s_1t_1^2t_2^4
\\ &+& z^2s_2t_1^4t_2^2 + 6zd_1d_2s_1s_2 + zd_1d_2s_1t_2^2 + zd_1d_2s_2t_1^2 + 2zd_1d_2t_1^2t_2^2 - 4zd_1s_1s_2^2t_2
\\ &+& zd_1s_1s_2t_2^3 - 2zd_1s_2^2t_1^2t_2 - 4zd_2s_1^2s_2t_1 - 2zd_2s_1^2t_1t_2^2 + zd_2s_1s_2t_1^3 + 8zs_1^2s_2^2t_1t_2
\\ &-& 2zs_1^2s_2t_1t_2^3 - 2zs_1s_2^2t_1^3t_2 + 9d_1^2d_2^2 - 6d_1^2d_2s_2t_2 + d_1^2d_2t_2^3 + d_1^2s_2^3 - 6d_1d_2^2s_1t_1 + d_1d_2^2t_1^3
\\ &-& 2d_1d_2s_1s_2t_1t_2 + 2d_1s_1s_2^3t_1 + d_2^2s_1^3 + 2d_2s_1^3s_2t_2 - 4s_1^3s_2^3 + s_1^3s_2^2t_2^2 + s_1^2s_2^3t_1^2.
\end{eqnarray*}

\begin{center} {\bf Acknowledgements}\end{center}

This project was caried out in the framework of the exchange program
between the Hungarian and Bulgarian Academies of Sciences.
We thank for this support.

\goodbreak


\end{document}